\documentclass[11pt]{article}
\usepackage{amsfonts}
\usepackage{amssymb,amsmath}
\usepackage{amsthm}
\usepackage{latexsym}
\usepackage{graphics}
\usepackage{epic}
\usepackage{epsfig}
\usepackage{psfrag}
\usepackage{bbm}
\usepackage{dsfont}
\usepackage{enumitem}
\usepackage{stmaryrd}
\usepackage{color}

\numberwithin{equation}{section}
\def\PP{\mathbb{P}}

\def\RR{\mathbb{R}}
\def\NN{\mathbb{N}}

\def\EE{\mathbb{E}}
\def\11{\mathbbm{1}}

\def\E{\mathbb{E}}
\def\P{\mathbb{P}}

\def\N{\mathbb{N}}

\def\d{\partial}
\def\Z{\mathbb{Z}}
\def\ZZ{\mathbb{Z}}

\newtheorem{thm}{Theorem}[section]

\newtheorem{prop}[thm]{Proposition}

\newtheorem{hyp}{Assumption}

\theoremstyle{remark}
\newtheorem{rem}{Remark}

\begin{document}

\title{Quasi-stationary distribution for multi-dimensional birth and death processes conditioned to survival of all coordinates}

\author{Nicolas Champagnat$^{1,2,3}$, Denis Villemonais$^{1,2,3}$}

\footnotetext[1]{Universit\'e de Lorraine, IECL, UMR 7502, Campus Scientifique, B.P. 70239,
  Vand{\oe}uvre-l\`es-Nancy Cedex, F-54506, France}
\footnotetext[2]{CNRS, IECL, UMR 7502,
  Vand{\oe}uvre-l\`es-Nancy, F-54506, France} \footnotetext[3]{Inria, TOSCA team,
  Villers-l\`es-Nancy, F-54600, France.\\
  E-mail: Nicolas.Champagnat@inria.fr, Denis.Villemonais@univ-lorraine.fr}

\maketitle

\begin{abstract}
  This article studies the quasi-stationary behaviour of multidimensional birth and death processes, modeling the interaction between
  several species, absorbed when one of the coordinates hits 0. We study models where the absorption rate is not uniformly bounded,
  contrary to most of the previous works. To handle this natural situation, we develop original Lyapunov function arguments that
  might apply in other situations with unbounded killing rates. We obtain the exponential convergence in total variation of the
  conditional distributions to a unique stationary distribution, uniformly with respect to the initial distribution. Our results
  cover general birth and death models with stronger intra-specific than inter-specific competition, and cases with neutral
  competition with explicit conditions on the dimension of the process.
\end{abstract}



\section{Introduction}
\label{sec:intro}


This article is devoted to the study of quasi-stationary behavior of multi-type birth and death processes absorbed when one of the
types goes extinct. 

More specifically, we consider a continuous-time Markov process $(X_t,t\geq 0)$ taking values in $\ZZ_+^r$ for some $r\geq 1$, where
we use the notations $\ZZ_+=\{0,1,2,\ldots\}$ and $\NN=\{1,2,3,\ldots\}$. The transition rates of $X$ are given by
$$
\text{from }n=(n_1,\ldots,n_r)\text{ to }
\begin{cases}
  n+e_j & \text{with rate }n_j b_j(n),\\ 
  n-e_j & \text{with rate }n_j\left[d_j(n)+\Big(\sum_{k=1}^r c_{jk}(n) n_k\Big)^\gamma\right]
\end{cases}
$$
for all $1\leq j\leq r$, with $e_j=(0,\ldots,0,1,0,\ldots,0)$, where the nonzero coordinate is the $j$-th one, $\gamma>0$,
$b(n)=(b_1(n),\ldots,b_r(n))$ and $d(n)=(d_1(n),\ldots,d_r(n))$ are functions from $\ZZ_+^r$ to $\RR_+^r$ and
$c(n)=(c_{ij}(n))_{1\leq i,j\leq r}$ is a function from $\ZZ_+^r$ to the set of $r\times r$ matrices with nonnegative coefficients.

In other words, the infinitesimal generator of the process $X$ is defined for all bounded function $f$ on $\ZZ^r_+$ and all $n\in
\ZZ_+^r$ as
\begin{multline}
  Lf(n)=\sum_{j=1}^r[f(n+e_j)-f(n)]n_j b_j(n) \\ +\sum_{j=1}^r[f(n-e_j)-f(n)]n_j\left[d_j(n)+\left(\sum_{k=1}^r c_{jk}(n) n_k\right)^\gamma\right]
  \label{eq:generator}
\end{multline}

This model represents a density-dependent population dynamics with $r$ types of individuals (say $r$ species), where $b_j(n)$ (resp.\
$d_j(n)$\,) is the individual birth (resp.\ death) rate of an individual of type $j$ in the population $n$, and $c_{ij}(n)$
represents the competition exerted by an individual of type $j$ on an individual of type $i$ in the population $n$. The global
competition $\sum_{k=1}^r c_{jk}(n) n_k$ felt by an individual of type $i$ influences its death rate at a power $\gamma$, which
represents the strength of the competition. The larger $\gamma$ is, the stronger is the influence of the competition in large
populations. The case $\gamma=1$ is very common in biology and is known as logistic (or competitive Lotka-Volterra) competition.
Other values of $\gamma>0$ are also relevant in ecological applications.

Note that the forms of the birth and death rates imply that, once a coordinate $X^j_t$ of the process hits 0, it remains equal to 0.
This corresponds to the extinction of the population of type $j$. Hence, the set $\d:=\ZZ_+^r\setminus\NN^r$ is absorbing for the
process $X$. Let us denote by $\tau_\d$ its absorption time. Our goal is to study the process $(X_t,t\geq 0)$ conditioned to
non-aborption, and in particular its quasi-stationary distribution, i.e.\ a probability measure $\alpha$ on $\NN^r$ such that
$$
\PP_\alpha(X_t\in\cdot\mid t<\tau_\d)=\alpha,\quad\forall t\geq 0.
$$
More precisely, we shall give conditions ensuring the uniform exponential convergence of conditional distributions to the
quasi-stationary distribution, independently of the initial condition. This means that there exist constants $C,\lambda>0$ such
that
\begin{equation}
  \label{eq:conv}
  \left\|\PP_\mu(X_t\in\cdot\mid t<\tau_\d)-\alpha\right\|_{TV}\leq C e^{-\lambda t},\quad\forall \mu\in\mathcal{P}(\NN^r),\quad t\geq 0,
\end{equation}
where $\|\cdot\|_{TV}$ is the total variation norm and $\mathcal{P}(\NN^r)$ is the set of probability measures on $\NN^r$. This
implies in particular the uniqueness of the quasi-stationary distribution.

In particular, when~\eqref{eq:conv} is satisfied, regardless of the initial condition, the quasi-stationary distribution describes
the state of the population when it survives for a long time. One of the most notable features of quasi-stationary populations is the
existence of a so-called \emph{mortality/extinction plateau}: there exists $\lambda_0>0$ limit of the extinction rate of the
population (see~\cite{meleard-villemonais-12}). The constant $-\lambda_0$ is actually the largest non-trivial eigenvalue of the
generator $L$ and satisfies
$$
\PP_\alpha(t<\tau_\d)=e^{-\lambda_0 t},\quad\forall t\geq 0.
$$

Many properties can be deduced from~\eqref{eq:conv}, as the uniform convergence of $e^{\lambda_0 t}\PP_x(t<\tau_\d)$ to $\eta(x)$,
where $\eta$ is the eigenfunction of $L$ corresponding to the eigenvalue $\lambda_0$~\cite[Prop.\,2.3]{ChampagnatVillemonais2014},
and the existence and the exponential ergodicity of the associated $Q$-process, defined as the process $X$ conditionned to never be
extinct (see~\cite[Thm.\,3.1]{ChampagnatVillemonais2014} for a precise definition).

Quasi-stationary distributions for population processes have received much interest in the recent years (see the
surveys~\cite{meleard-villemonais-12,vanDoorn2013}). One of the most understood case concerns birth and death processes on $\Z_+$
absorbed at $0$: it has been shown in~\cite{vanDoorn1991} that there are either zero, one or an infinite continuum of
quasi-stationary distributions for such processes. More recently, it has been shown in~\cite{Martinez-Martin-Villemonais2012} that
there exists a unique quasi-stationary distribution for one dimensional birth and death processes if and only if~\eqref{eq:conv}
holds. This result has been extended in the recent paper~\cite{ChampagnatVillemonais2014} to birth and death processes with
catastrophes (among other applications). The specific question of estimates on the speed of convergence to quasi-stationary
distributions for one dimensional birth and death processes has been studied in~\cite{diaconis-miclo-09,chazottes-al-15}. The
multi-dimensional situation, which takes into account the existence of several types of individuals in a population, is much less
understood, except in the branching case of multi-type Galton-Watson processes (see~\cite{Athreya1972,penisson-11}) and for specific
cooperative models with bounded absorption rate (see~\cite{ChampagnatVillemonais2014}). Another originality of the models studied in
the present paper is the fact that the absorption rate is not uniformly bounded. To handle this natural situation, we develop
original Lyapunov function arguments that might apply to other situations with unbounded killing rates. For results on the
quasi-stationary behaviour of continuous-time and continuous state space models, we refer
to~\cite{CCLMMS09,Littin2012,ChampagnatVillemonais2015} for the one dimensional case and
to~\cite{Pinsky1985,Gong1988,Cattiaux2008,KnoblochPartzsch2010, DelMoralVillemonais2015, ChampagnatCoulibalyVillemonais2015} for the
multi-dimensional situation. Infinite dimensional models have been studied in~\cite{Collet2011,ChampagnatVillemonais2014}. Several
papers studying the quasi-stationary behaviour of models applied to biology, chemistry, demography and finance are listed
in~\cite{Pollett}.

We are going to prove~\eqref{eq:conv} under two sets of assumptions. The first one (Section~\ref{sec:one}) considers stronger
intra-specific competition than inter-specific competition (i.e.\ $c_{ii}(n)$ larger than $c_{ij}(n)$ for large $|n|$ and for $i\neq
j$). We make no particular assumption on the dimension $r$ of the process and on the birth, death and competition functions. The
second one (Section~\ref{sec:two}) considers a case of equal inter- and intra-specific competition (neutral competition), which leads
to specific difficulties that we can solve for dimension $r\leq 3$ in the logistic case. The last section~\ref{sec:ext} is dedicated
to a few extensions of our methods to other models.

\section{General birth and death functions with strong intra-specific competition}
\label{sec:one}

The first case we study corresponds to the following assumptions, where $|n|$ denotes $n_1+\ldots+n_r$, for all $n=(n_1,\ldots,n_r)$.

\begin{hyp}
  \label{hyp1}  
  \begin{description}
  \item[\textmd{(H1)}] There exist constants $\bar b$, $\bar d$ and $\underline{c}$ in $(0,\infty)$ and $\beta_1\geq 0,\beta_2\in(-\infty,1)$ with
    $\beta_1+\gamma\beta_2<\gamma$ such that, for all $n\in\NN^r$ and $i\in\{1,\ldots,r\}$,
    $$
    0<b_i(n)\leq \bar b|n|^{\beta_1},\quad 0\leq d_i(n)\leq\bar d|n|^{\beta_1},\quad c_{ii}(n)\geq\underline{c}|n|^{-\beta_2}.
    $$
  \item[\textmd{(H2)}] For all $i\in\{1,\ldots,r\}$, when $|n|\rightarrow+\infty$,
    \begin{equation}
      \label{eq:hyp}
      c_{ii}(n)\gg\sum_{1\leq j\neq k\leq r} c_{jk}(n)+\frac{1}{|n|}\sum_{j=1}^r c_{jj}(n),
    \end{equation}
    where the notation $f(n)\gg g(n)$ means that $f(n)/g(n)\rightarrow+\infty$ when $|n|\rightarrow+\infty$.
  \end{description}
\end{hyp}

The first assumption is standard for population models, typically with $\beta_1=\beta_2=0$. From the biological point of view, the
second assumption corresponds to a stronger intra-specific competition than interspecific competition. For example it holds if all
the functions $c_{ii}(\cdot)$ for $1\leq i\leq r$ have the same order of magnitude when $|n|\rightarrow+\infty$ and all the
$c_{ij}(\cdot)$ for $i\neq j$ are asymptotically negligible w.r.t.\ the $c_{ii}(\cdot)$. The positivity of $b_i(n), c_{ii}(n)$ ensures
that the process is irreducible away from $\d$ (in the sense that $\P_n(X_1=m)>0$ for all $n,m\neq\d$).

\begin{thm}
  \label{thm:one}
  Under Assumptions~\ref{hyp1}, there exist constants $C,\lambda>0$ such that~\eqref{eq:conv} holds true.
\end{thm}

We explain in Section~\ref{sec:ext} how to generalize this result to cases with killing or with multiple births.

\begin{rem}
  \label{rem:one}
  In fact, Assumption~(H2) and the statement about $c_{ii}(n)$ in Assumption~(H1) can be replaced by
  $$
  \sum_{j=1}^r \frac{n_j}{|n|}\mathbbm{1}_{n_j\neq 1}\left(\sum_{k=1}^r c_{jk}(n)n_k\right)^\gamma\geq C_r
  \sum_{j=1}^r \mathbbm{1}_{n_j=1}\left(\sum_{k=1}^r c_{jk}(n)n_k\right)^\gamma\gg|n|^{\beta_1\vee\gamma}
  $$
  when $|n|$ large enough, for some (explicit) constant $C_r$ depending only on $r$ and $\gamma$.

  This allows to cover other biological settings than the strong intra-specific competition. For example, assuming $\beta_1=0$, it is
  satisfied when, for all $i$, $c_{ij}(n)=n_i^{1+\delta}$ for at least some $j\neq i$ for some $\delta>0$, and $c_{ij}(n)=1$
  otherwise. This corresponds to a situation of strong competition exerted on other species by large species (a kind of collective
  aggressivity against other species). Hybrid situations can also fit this assumption, for example when $r=2$,
  $c_{12}(n)=n_1^{1+\delta}$, $c_{22}(n)=n_2^{\delta'}$ and $c_{11}(n)=c_{21}(n)=1$ for some $\delta,\delta'>0$. More complex
  structures of interaction can of course fit our assumptions.

  This also allows to cover classical biological cases, with comparable inter- and intra-specific competition, for example if
  $\gamma=1$ and $c_{ij}(n)=c_{ij}$ independent of $n$, provided that the $c_{ii}$ for $1\leq i\leq r$ are not too small compared to
  the $c_{ij}$, $i\neq j$.
\end{rem}

This result and the one of Section~\ref{sec:two} are consequences of the general criterion
of~\cite[Thm.\,2.1]{ChampagnatVillemonais2014}, which gives a necessary and sufficient condition for~\eqref{eq:conv} for general
Markov processes. This condition is given by the two properties (A1) and (A2) below: there exists a probability measure $\nu$ on
$\NN^r$ such that
\begin{itemize}
\item[(A1)] there exist $t_0,c_1>0$ such that for all $x\in\NN^r$,
  $$
  \PP_x(X_{t_0}\in\cdot\mid t_0<\tau_\d)\geq c_1\nu(\cdot);
  $$
\item[(A2)] there exists $c_2>0$ such that for all $x\in\NN^r$ and $t\geq 0$,
  $$
  \PP_\nu(t<\tau_\d)\geq c_2\PP_x(t<\tau_\d).
  $$
\end{itemize}

Hence we only have to prove that Assumptions~\ref{hyp1} implies~(A1) and~(A2). Our proof of~(A1) is based on Lyapunov functions and
makes use of the following general inequality on conditional distributions.

\begin{prop}
  \label{prop:mu_t}
  Fix $n\in\NN^r$ and let $\mu_t(\cdot)=\PP_n(X_t\in\cdot\mid t<\tau_\d)$. Let $V:\NN^r\rightarrow\RR_+$ such that $LV$ is bounded
  from above on $\NN^r$. Then, for all $t\geq 0$,
  \begin{equation}
    \label{eq:prop-mu_t}
    \mu_t(V)-V(n)\leq \int_0^t\Big[\mu_s(LV)-\mu_s(V)\mu_s(L\mathbbm{1}_{\NN^r})\Big]\,ds,    
  \end{equation}
  where the value of the integral in the r.h.s.\ is well-defined in $(-\infty,+\infty]$ since
  $$
  \int_0^t\Big[\mu_s(LV)-\mu_s(V)\mu_s(L\mathbbm{1}_{\NN^r})\Big]_-\,ds<\infty,
  $$
  where $[x]_-=(-x)\vee 0$ is the negative part of $x\in\RR$.
\end{prop}

\begin{proof}
  Fix $k\in\NN$ and let $\tau_k:=\inf\{t\geq 0:|X_t|\geq k\}$. We define $X^k$ as the process $X$ stopped at time $\tau_k$, and
  denote by $L^k$ its infinitesimal generator, given by $L^k f(n)=Lf(n)\mathbbm{1}_{|n|<k}$. Dynkin's formula then entails
  $$
  \EE_n V(X^k_t)=V(n)+\int_0^t\EE_n\left[L^kV(X^k_s)\right]\,ds.
  $$
  Letting $k\rightarrow+\infty$, Fatou's lemma applied to both sides (mind that $V$ is bounded from below while $LV$ is bounded from
  above) imply that
  \begin{equation}
    \label{eq:prop1}
    \EE_n V(X_t)\leq V(n)+\int_0^t\EE_n\left[LV(X_s)\right]\,ds.
  \end{equation}
  Similarly,
  $$
  \EE_n\mathbbm{1}_{\NN^r}(X^k_t)=1+\int_0^t\EE_n\left[L^k\mathbbm{1}_{\NN^r}(X^k_s)\right]\,ds.
  $$
  Lebesgue's dominated convergence theorem applies to the l.h.s.\ and the monotone convergence theorem to the r.h.s., which entails
  $$
  \EE_n\mathbbm{1}_{\NN^r}(X_t)=1+\int_0^t\EE_n\left[L\mathbbm{1}_{\NN^r}(X_s)\right]\,ds.
  $$
  Now, fix $T\geq 0$.
  We have $1\geq \PP_n(t<\tau_\d)\geq \PP_n(T<\tau_\d)>0$ for all $t\in[0,T]$. Note first that, since $LV$ is bounded from above and
  $V$ is non-negative,~\eqref{eq:prop1} implies that $t\mapsto\EE_n[LV(X_t)]\in L^1([0,T])$. In addition, $L\mathbbm{1}_{\NN^r}\leq
  0$ and $V\geq 0$, so $-\EE_n(V(X_t))\EE_n(L\mathbbm{1}_{\NN^r}(X_t))\geq 0$. Then, either
  $$
  \int_0^t\EE_n(V(X_s))\EE_n(-L\mathbbm{1}_{\NN^r}(X_s))\,ds=+\infty,
  $$
  and then~\eqref{eq:prop-mu_t} is trivial since $\mu_t(f)=\EE_n(f(X_t))/\EE_n(\mathbbm{1}_{\NN^r}(X_t))\geq\EE_n(f(X_t))$ for all
  $f\geq 0$, or
  $$
  \int_0^t\EE_n(V(X_s))\EE_n(-L\mathbbm{1}_{\NN^r}(X_s))\,ds<+\infty.
  $$
  In this case, since $\mu_t(f)=\EE_n(f(X_t))/\EE_n(\mathbbm{1}_{\NN^r}(X_t))\leq\EE_n(f(X_t))/\PP_n(T<\tau_\d)$ for all $f\geq 0$,
  we deduce from the fundamental theorem of calculus that, for all $t\in[0,T]$,
  $$
  \int_0^t\Big[\mu_s(LV)-\mu_s(V)\mu_s(L\mathbbm{1}_{\NN^r})\Big]\,ds=\frac{V(n)+\int_0^t\EE_n\left[LV(X_s)\right]\,ds}
  {1+\int_0^t\EE_n\left[L\mathbbm{1}_{\NN^r}(X_s)\right]\,ds}-V(n).
  $$
  Inequality~\eqref{eq:prop-mu_t} then follows from~\eqref{eq:prop1}.
\end{proof}

\begin{proof}[Proof of Theorem~\ref{thm:one}]
The proof of~(A1) is based on the following bounded Lyapunov function: fix $\varepsilon\in(0,\gamma-\gamma\beta_2)$ and define for
all $n\in\NN^r$
$$
V_\varepsilon(n)=\sum_{j=1}^{|n|}\frac{1}{j^{1+\varepsilon}},
$$
and $V_\varepsilon(n)=0$ for all $n\in\d:=\ZZ_+^r\setminus\NN^r$. For all $m, n\in\NN^r$ such that $|m|\leq|n|$, we have in
particular the inequality
\begin{multline}
  \label{eq:bound-V}
  \frac{1}{\varepsilon}\left(\frac{1}{(|m|+1)^\varepsilon}-\frac{1}{(|n|+1)^\varepsilon}\right)=\int_{|m|+1}^{|n|+1}\frac{dx}{x^{1+\varepsilon}} \\
  \leq V_\varepsilon(n)-V_\varepsilon(m)
  \leq\int_{|m|}^{|n|}\frac{dx}{x^{1+\varepsilon}}=\frac{1}{\varepsilon}\left(\frac{1}{|m|^\varepsilon}-\frac{1}{|n|^\varepsilon}\right).
\end{multline}
\medskip

\noindent\textit{Step 1: Lyapunov function for the conditional distributions.}\\
Fix $n_0\in\NN^r$ and let $\mu_t(\cdot):=\PP_{n_0}(X_t\in\cdot\mid t<\tau_\d)$. It follows from~\eqref{eq:generator} that
\begin{multline}
 \mu_t(LV_\varepsilon)=\sum_{n\in\NN^r} \mu_t(n) \sum_{i=1}^r \left\{\mathbbm{1}_{n_i\neq 1}
 \left(\frac{n_ib_i(n)}{(|n|+1)^{1+\varepsilon}}-\frac{n_i\left[d_i(n)+\left(\sum_{j=1}^r c_{ij}(n)
         n_j\right)^\gamma\right]}{|n|^{1+\varepsilon}}\right) \right. \\
 \left.+\mathbbm{1}_{n_i=1}\left(\frac{b_i(n)}{(|n|+1)^{1+\varepsilon}}-\left[d_i(n)+\left(\sum_{j=1}^r c_{ij}(n)
       n_j\right)^\gamma\right]V_\varepsilon(n)\right)\right\} \label{eq:proof-1}
\end{multline}
Fix $n\in\NN^r$ and let $i^*(n)$ be (one of) the argmax of $i\mapsto n_i$ and let $c^*(n):=c_{i^*(n)i^*(n)}(n)$. Then $n_{i^*(n)}\geq |n|/r$ and
$$
n_{i^*(n)}\left(\sum_{j=1}^r c_{ij}(n) n_j\right)^\gamma\geq
\frac{c^*(n)^\gamma |n|^{1+\gamma}}{r^{1+\gamma}},
$$
Using Assumption~(H1) and since $d_i(n)\geq 0$,
\begin{multline}
  \mu_t(LV_\varepsilon)\leq\sum_{n\in\NN^r}\mu_t(n)\left\{\bar b
    |n|^{\beta_1-\varepsilon}-\mathbbm{1}_{n\neq(1,\ldots,1)}
    \frac{c^*(n)^\gamma |n|^{\gamma-\varepsilon}}{r^{1+\gamma}}
    \right. \\ \left.-\sum_{i=1}^r\mathbbm{1}_{n_i=1}\left[d_i(n)+\left(\sum_{j=1}^r c_{ij}(n) n_j\right)^\gamma\right]
    V_\varepsilon(n)\right\}. \label{eq:proof-1-bis}
\end{multline}
In addition,
$$
-\mu_t(V_\varepsilon)\mu_t(L\mathbbm{1}_{\NN^r})\leq\|V_\varepsilon\|_\infty\sum_{n\in\NN^r}\mu_t(n)
\sum_{i=1}^r\mathbbm{1}_{n_i=1}\left[d_i(n)+\left(\sum_{j=1}^r c_{ij}(n) n_j\right)^\gamma\right].
$$
Denoting $a:=\frac{c^*(1,\ldots,1)^\gamma}{r^{1+\varepsilon}} $, the last two equations imply
$$
\mu_t(LV_\varepsilon) -\mu_t(V_\varepsilon)\mu_t(L\mathbbm{1}_{\NN^r})\leq \sum_{n\in\NN^r}\mu_t(n)\left[\bar b
  |n|^{\beta_1-\varepsilon}+a-\frac{c^*(n)^\gamma
    |n|^{\gamma-\varepsilon}}{r^{1+\gamma}}\right]+A,
$$
where, by~\eqref{eq:bound-V},
\begin{align*}
  A & :=\sum_{n\in\NN^r}\mu_t(n)\sum_{i=1}^r\mathbbm{1}_{n_i=1}\left[d_i(n)+\left(\sum_{j=1}^r
      c_{ij}(n) n_j\right)^\gamma\right]\left(\|V_\varepsilon\|_\infty-V_\varepsilon(n)\right) \\ 
  & \leq\sum_{n\in\NN^r}\mu_t(n)\sum_{i=1}^r\mathbbm{1}_{n_i=1}\frac{\bar d
    |n|^{\beta_1}+\left(\sum_{j\neq i}c_{ij}(n)|n|+c_{ii}(n)\right)^\gamma}{\varepsilon|n|^\varepsilon} \\
  & \leq\sum_{n\in\NN^r}\mu_t(n)\left\{\frac{r\bar d |n|^{\beta_1-\varepsilon}}{\varepsilon}+
    \frac{C_{\gamma,r}|n|^{\gamma-\varepsilon}}{\varepsilon}\left(\sum_{1\leq i\neq j\leq r}
      c_{ij}(n)+\frac{1}{|n|}\sum_{i=1}^r c_{ii}(n)\right)^\gamma\right\},
\end{align*}
where $C_{\gamma,r}$ is a positive constant such that $x_1^\gamma+\ldots+x^\gamma_r\leq C_{\gamma,r}(x_1+\ldots+x_r)^\gamma$ for all $x_1,\ldots,x_r\geq 0$.

Using Assumption~(H2), we see that there exist constants $B>0$ and $C>1$ independent of $n$ such that
\begin{align*}
  \mu_t(LV_\varepsilon) -\mu_t(V_\varepsilon)\mu_t(L\mathbbm{1}_{\NN^r}) & \leq \sum_{n\in\NN^r}\mu_t(n)
  \left\{B(1+|n|^{\beta_1-\varepsilon})-\frac{c^*(n)^\gamma|n|^{\gamma-\varepsilon}}{2 r^{1+\gamma}}\right\} \\
  & \leq \sum_{n\in\NN^r}\mu_t(n)
  \left\{B(1+|n|^{\beta_1-\varepsilon})-\frac{\underline{c}^\gamma}{2 r^{1+\gamma}}|n|^{\gamma-\gamma\beta_2-\varepsilon}\right\} \\
  & \leq \left(C-\frac{1}{C}\sum_{n\in\NN^r}|n|^{\gamma-\gamma\beta_2-\varepsilon}\mu_t(n)\right),
\end{align*}
where the last inequality follows from the fact that $\beta_1<\gamma-\gamma\beta_2$ and $\varepsilon<\gamma-\gamma\beta_2$.
\medskip

\noindent\textit{Step 2: Proof of~(A1).}\\
Step~1 and Prop.~\ref{prop:mu_t} imply that
$$
\mu_t(V_\varepsilon)\leq \frac{1}{\varepsilon}+\int_0^t
\left(C-\frac{1}{C}\sum_{n\in\NN^r}|n|^{\gamma-\gamma\beta_2-\varepsilon}\mu_s(n)\right)\,ds.
$$
This implies that, for any initial condition $n_0$, there exists $s\leq 1/(\varepsilon C)$ such that
$$
\sum_{n\in\NN^r}|n|^{\gamma-\gamma\beta_2-\varepsilon}\mu_s(n)\leq 2C^2.
$$
Let us define the set $K= \{n\in\NN^r,\ |n|^{\gamma-\gamma\beta_2-\varepsilon}< 4C^2\}$, which is finite since $\varepsilon\in(0,\gamma-\gamma\beta_2)$.
 Using the previous inequality and Markov's inequality, we obtain that
$
\mu_s(K)\geq 1/2. 
$

The minimum $p=\min_{x\in K} \P_x(X_u=(1,\ldots,1),\ \forall u\in[1/\varepsilon C,2/\varepsilon C))$ is positive because $K$ is
finite, $X$ is irreducible away from $\d$ and the jumping rate from $(1,\ldots,1)$ is finite. Hence
\begin{align*}
\mu_s\Big(\P_\cdot(X_u=(1,\ldots,1),\ \forall u\in[1/\varepsilon C,2/\varepsilon C))\Big)\geq \frac{p}{2} >0,
\end{align*}
  Using the Markov property, we deduce that
\begin{align*}
\mu_{2/(\varepsilon C)}\{(1,\ldots,1)\}&\geq \mu_s(\P_\cdot(X_{2/(\varepsilon C)-s}=(1,\ldots,1))\\
&\geq \mu_s(\P_\cdot(X_u=(1,\ldots,1),\ \forall u\in[1/(\varepsilon C),2/\varepsilon C))\\
&\geq \frac{p}{2}>0.
\end{align*}
Since $p$ does not depend on the initial distribution of the process, we deduce that (A1) is satisfied with $\nu=\delta_{(1,\ldots,1)}$, $t_0=2/(\varepsilon C)$ and $c_1=p/2$.

\medskip
\noindent\textit{Step 3: Proof of~(A2).}\\
The same calculation as in~\eqref{eq:proof-1-bis} shows that, for all $n\in \NN^r$,
\begin{align*}
LV_{\varepsilon}(n)\leq \bar b
    |n|^{\beta_1-\varepsilon}-\mathbbm{1}_{n\neq(1,\ldots,1)}
    \frac{c^*(n)^\gamma |n|^{\gamma-\varepsilon}}{r^{1+\gamma}}.
\end{align*}
Now, using Assumption~(H1), we have $c^*(n)^\gamma |n|^{\gamma-\varepsilon}\geq \underline{c} |n|^{\gamma-\varepsilon-\gamma\beta_2}$, with
$\gamma-\varepsilon-\gamma\beta_2>(\beta_1-\varepsilon)\vee 0$. Hence, there exist two positive constants $C_1,C_2>0$ such that
\begin{align*}
LV_{\varepsilon}(n)\leq C_1-C_2|n|^{\gamma-\varepsilon-\gamma\beta_2},\ \forall n\in\N^r.
\end{align*}
Since $V_\varepsilon$ is bounded and $LV_\varepsilon$ is bounded from above, we deduce from  Dynkin's formula as in the proof of Proposition~\ref{prop:mu_t} that, for all $k\geq 1$,
\begin{align*}
\E_n\left(V_\varepsilon(X_{\tau_{\{|m|\leq k\}}\wedge\tau_\d})\right)&\leq V_\varepsilon(n)+\E\left(\int_0^{\tau_{\{|m|\leq k\}}\wedge\tau_\d}LV_{\varepsilon}(X_t)\,dt\right)\\
&\leq \frac{1}{\varepsilon}+\left(C_1-C_2|k|^{\gamma-\varepsilon-\gamma\beta_2}\right)\E_n\left(\tau_{\{|m|\leq k\}}\wedge\tau_\d\right),
\end{align*}
where $\tau_{\{|m|\leq k\}}$ is the first hitting time by $X_t$ of the set $\{m\in\NN^r:|m|\leq k\}$
Since
$V_\varepsilon(X_{\tau_{\{|m|\leq k\}}\wedge\tau_\d})$ is almost surely non-negative, we obtain
\begin{align*}
\sup_{n\in\NN^r}\E_n\left(\tau_{\{|m|\leq k\}}\wedge\tau_\d\right)\xrightarrow[k\rightarrow\infty]{} 0.
\end{align*}
Using the same argument as in~\cite[Eq.~(4.6)]{ChampagnatVillemonais2014}, we deduce that, for all $\lambda>0$, there exists $k\geq 1$ such that
\begin{equation}
  \label{eq:moment-expo}
  \sup_{n\in\N^r}\E_n(e^{\lambda \tau_{\{|m|\leq k\}}\wedge\tau_\d})<+\infty.
\end{equation}
Let us now denote by $\lambda$ the total jumping rate from $(1,\ldots,1)$. We choose $k$ such that~\eqref{eq:moment-expo} holds for
this constant $\lambda$. Defining the finite set $G=\{m\in\N^r \mid |m|\leq k\}$, we thus have
\begin{align}
\label{eq:expo-moment}
A:=\sup_{n\in\N^r}\E_n(e^{\lambda \tau_G\wedge \tau_\d})<\infty.
\end{align}
The irreducibility of $X$ and the finiteness of $G$ entail the existence of a constant $C>0$ such that
\begin{equation}
    \label{eq:ineq-PB-catastrophe}
    \sup_{n\in G}\P_n(t<\tau_\d)\leq C\inf_{n\in G}\P_n(t<\tau_\d),\quad\forall t\geq 0.
  \end{equation}
For all $n\in \N^r$, we deduce from Chebyshev's inequality and~\eqref{eq:expo-moment} that 
\begin{align}
\label{eq:domination}
\P_n(t<\tau_G\wedge \tau_\d)\leq Ae^{-\lambda t}.
\end{align}
%
 Using the last two inequalities and the strong Markov property, we have
\begin{align*}
    \P_n(t<\tau_\d) & =\P_n(t<\tau_G\wedge\tau_\d)+\P_n(\tau_G\wedge\tau_\d\leq t<\tau_\d)\\
&\leq Ae^{-\lambda t}+\int_0^t \sup_{m\in G\cup\{\d\}}\P_m(t-s<\tau_\d)\P_n(\tau_G\wedge\tau_\d\in ds)\\
&\leq Ae^{-\lambda t}+ C\int_0^t \P_{x_0}(t-s<\tau_\d)\P_n(\tau_G\wedge\tau_\d\in ds),
\end{align*}
where $x_0:=(1,\ldots,1)$.
Now, by definition of $\lambda$ and by the Markov property,
\begin{align*}
\P_{x_0}(t-s<\tau_\d)e^{-\lambda s}& =\P_{x_0}(t-s<\tau_\d)\P_{x_0}(X_u=x_0,\forall u\in [0,s])\\
&\leq \P_{x_0}(t<\tau_\d).
  \end{align*}
  Hence
\begin{align*}
\P_n(t<\tau_\d) &\leq Ae^{-\lambda t}+ C\P_{x_0}(t<\tau_\d)\int_0^t e^{\lambda s}\P_n(\tau_G\wedge\tau_\d\in ds).
\end{align*}  
We finally deduce that
\begin{align*}
\P_n(t<\tau_\d) &\leq A \,\P_{x_0}(t<\tau_\d)+CA\,\P_{x_0}(t<\tau_\d).
\end{align*}
This entails (A2) for $\nu=\delta_{x_0}=\delta_{(1,\ldots,1)}$. 
\end{proof}

\section{Birth and death processes with neutral competition in not too large dimension}
\label{sec:two}

The second case we study corresponds to the following assumptions.

\begin{hyp}
  \label{hyp2}
  \begin{description}
  \item[\textmd{(H3)}] There exist constants $\bar b$, $\bar d$ in $(0,\infty)$ and $\beta_1\in[0,\gamma)$ such
    that, for all $n\in\NN^r$ and $i\in\{1,\ldots,r\}$,
    $$
    0<b_i(n)\leq \bar b|n|^{\beta_1},\quad 0\leq d_i(n)\leq\bar d|n|^{\beta_1}.
    $$
  \item[\textmd{(H4)}] There exists a constant $c>0$ such that, for all $i,j\in\{1,\ldots,r\}$, $c_{ij}(n)=c$ (\emph{neutral competition}).
  \end{description}
\end{hyp}


\begin{thm}
  \label{thm:two}
  Under Assumptions~\ref{hyp2} and if $r< 1+e\gamma$, there exist constants $C,\lambda>0$ such that~\eqref{eq:conv} holds true.
\end{thm}

\begin{rem}
  In the classical logistic case $\gamma=1$, this gives the existence and exponential convergence in total variation to the
  quasi-stationary distribution up to dimension 3. The larger $\gamma$ is, the larger the dimension $r$ can be taken.
\end{rem}

\begin{proof}[Proof of Conditions~(A1) and~(A2) under Assumptions~\ref{hyp2}]

We proceed as in the proof of Theorem~\ref{thm:one}. Fix $\varepsilon\in(0,\gamma)$. First, we deduce from~\eqref{eq:proof-1} and from the fact that $n_i\mathbbm{1}_{n_i\neq 1}\geq n_i-1$ that
\begin{align*}
  \mu_t(LV_\varepsilon) & \leq\sum_{n\in\NN^r}\mu_t(n)\left\{\bar b
    |n|^{\beta_1-\varepsilon}-c^\gamma\sum_{i=1}^r(n_i-1)|n|^{\gamma-1-\varepsilon}\right. \\
  & \qquad\qquad\qquad \left. -\sum_{i=1}^r\mathbbm{1}_{n_i=1}\left[d_i(n)+c^\gamma|n|^\gamma\right]
    V_\varepsilon(n)\right\} \\
  & \leq\sum_{n\in\NN^r}\mu_t(n)\left\{\bar b
    |n|^{\beta_1-\varepsilon}+\frac{rc^\gamma}{|n|^{1+\varepsilon-\gamma}}-c^\gamma|n|^{\gamma-\varepsilon} \right. \\
  & \qquad\qquad\qquad \left. -\sum_{i=1}^r\mathbbm{1}_{n_i=1}\left[d_i(n)+c^\gamma|n|^\gamma\right]
    V_\varepsilon(n)\right\}.
\end{align*}

In addition,
$$
-\mu_t(V_\varepsilon)\mu_t(L\mathbbm{1}_{\NN^r})=\mu_t(V_\varepsilon)\sum_{n\in\NN^r}\mu_t(n)
\sum_{i=1}^r\mathbbm{1}_{n_i=1}\left[d_i(n)+c^\gamma|n|^\gamma\right].
$$
Hence,
\begin{equation}
  \label{eq:proof-v2}
  \mu_t(LV_\varepsilon) -\mu_t(V_\varepsilon)\mu_t(L\mathbbm{1}_{\NN^r})\leq\sum_{n\in\NN^r}\mu_t(n)\left[\bar b
    |n|^{\beta_1-\varepsilon}+\frac{r c^\gamma}{|n|^{1+\varepsilon-\gamma}}-c^\gamma|n|^{\gamma-\varepsilon}\right]+A,
\end{equation}
 where
\begin{align}
  A & := \sum_{n\in\NN^r}\mu_t(n)\sum_{i=1}^r\mathbbm{1}_{n_i=1}\left[d_i(n)+c^\gamma|n|^\gamma\right]
    \left(\mu_t(V_\varepsilon)-V_\varepsilon(n)\right) \notag \\
    & \leq\sum_{n\in\NN^r}\mu_t(n)\biggl[\bar d r|n|^{\beta_1}(\|V_\varepsilon\|_\infty-V_\varepsilon(n)) +c^\gamma
      r^{1+\gamma}\|V_\varepsilon\|_\infty \notag \\ &\qquad\qquad\qquad\left.+c^\gamma|n|^\gamma(r-1)
    \left(\mu_t(V_\varepsilon)-V_\varepsilon(n)\right)_+\right], \label{eq:proof-v2-2}
\end{align}
where the last two terms are obtained by distinguishing between the cases where $n=(1,\ldots,1)$ (and hence $|n|=r$) and
$n\neq(1,\ldots,1)$ (and hence $\sum\mathbbm{1}_{n_i=1}\leq r-1$). Now, defining
$W_\varepsilon(n):=\|V_\varepsilon\|_\infty-V_\varepsilon(n)$, it follows from~\eqref{eq:bound-V} that
\begin{multline*}
  \sum_{n\in\NN^r}\mu_t(n)|n|^\gamma \left(\mu_t(V_\varepsilon)-V_\varepsilon(n)\right)_+ \\
    \leq\frac{1}{\varepsilon^{\gamma/\varepsilon}}\sum_{n\in\NN^r}\mu_t(n)W_\varepsilon(n)^{-\gamma/\varepsilon}
  \left(W_\varepsilon(n)-\mu_t(W_\varepsilon)\right)_+\\
  \leq\frac{1}{\varepsilon^{\gamma/\varepsilon}}\sum_{n\in\NN^r}\mu_t(n)W_\varepsilon(n)^{1-\gamma/\varepsilon}
  \left(1-\frac{\mu_t(W_\varepsilon)}{W_\varepsilon(n)}\right)_+.
\end{multline*}
Defining
$$
w_\varepsilon(n):=\frac{\mu_t(W_\varepsilon)}{W_\varepsilon(n)},\quad\forall n\in\NN^r,
$$
we obtain
$$
\sum_{n\in\NN^r}\mu_t(n)|n|^\gamma \left(\mu_t(V_\varepsilon)-V_\varepsilon(n)\right)_+\leq
\frac{\sum_{n\in\NN^r}\mu_t(n)
w_\varepsilon(n)^{\gamma/\varepsilon-1}(1-w_\varepsilon(n))_+}{\varepsilon^{\gamma/\varepsilon}\,\mu_t(W_\varepsilon)^{\gamma/\varepsilon-1}}.
$$
Now, it is elementary to check that
$$
x^{\frac{\gamma-\varepsilon}{\varepsilon}}(1-x)\leq\frac{\varepsilon}{\gamma}\left(\frac{\gamma-\varepsilon}{\gamma}\right)^{\frac{\gamma-\varepsilon}{\varepsilon}},\quad
\forall x\geq 0
$$
and H\"older's inequality implies that
$$
1\leq\mu_t(W_\varepsilon)^{1-\varepsilon/\gamma}\mu_t\left(W_\varepsilon^{-\frac{\gamma-\varepsilon}{\varepsilon}}\right)^{\varepsilon/\gamma}.
$$
Hence, it follows from~\eqref{eq:bound-V} that
\begin{align*}
  \sum_{n\in\NN^r}\mu_t(n)|n|^\gamma \left(\mu_t(V_\varepsilon)-V_\varepsilon(n)\right)_+ & \leq
  \frac{\mu_t\left(W_\varepsilon^{-\frac{\gamma-\varepsilon}{\varepsilon}}\right)}{\gamma\varepsilon^{\gamma/\varepsilon-1}}
  \left(\frac{\gamma-\varepsilon}{\gamma}\right)^{\frac{\gamma-\varepsilon}{\varepsilon}} \\
  & \leq\frac{1}{\gamma}\left(1-\frac{\varepsilon}{\gamma}\right)^{\frac{\gamma}{\varepsilon}-1}
  \sum_{n\in\NN^r}(|n|+1)^{\gamma-\varepsilon}\mu_t(n).
\end{align*}
Since
$$
\lim_{x\rightarrow 0}\left(1-x\right)^{\frac{1}{x}-1}=\frac{1}{e},
$$
under the assumption that $r<1+e\gamma$, we can find $\varepsilon>0$ small enough such that, for some $\delta>0$,
$$
(r-1)\sum_{n\in\NN^r}\mu_t(n)|n|^\gamma \left(\mu_t(V_\varepsilon)-V_\varepsilon(n)\right)_+
\leq (1-\delta) \sum_{n\in\NN^r}(|n|+1)^{\gamma-\varepsilon}\mu_t(n).
$$
Combining this with~\eqref{eq:proof-v2} and~\eqref{eq:proof-v2-2}, since
$\beta_1-\varepsilon<\gamma-\varepsilon$ and $\varepsilon<\gamma$, there exists a constant $C>0$ such that
\begin{align*}
  \mu_t(LV_\varepsilon) -\mu_t(V_\varepsilon)\mu_t(L\mathbbm{1}_{\NN^r}) & \leq\sum_{n\in\NN^r}\mu_t(n)\left[\left(\bar b+\frac{\bar d
        r}{\varepsilon}\right)|n|^{\beta_1-\varepsilon}+c^\gamma|n|^{\gamma-\varepsilon-1}\right. \\
  & \left.+\frac{c^\gamma r^{1+\gamma}}{\varepsilon}
    -c^\gamma|n|^{\gamma-\varepsilon}+c^\gamma(1-\delta)(|n|+1)^{\gamma-\varepsilon}\right] \\
  & \leq C-\frac{c^\gamma\delta}{2}\sum_{n\in\NN^r}|n|^{\gamma-\varepsilon}\,\mu_t(n).
\end{align*}
The rest of the proof is the same as for Thm.~\ref{thm:one}.

\end{proof}

\section{A few extensions to other models}
\label{sec:ext}
The method that we used is based on a Lyapunov type argument to prove conditions~(A1) and~(A2). This method is general enough to apply to a wide range of other models. We give here a few simple examples for which the exponential convergence of conditional distributions can be proved following the same arguments.

\bigskip\noindent\textbf{One dimensional birth and death processes with catastrophes.} We consider a standard birth an death process
on $\Z_+$ with birth (resp.\ death) rate $b_n$ (resp.\ $d_n$) from state $n\in\Z_+$, with $d_0=b_0=0$ and $b_n,d_n>0$ for all
$n\in\N$. This process is absorbed at $\d =0$. Moreover, we add a catastrophe rate $a_n\geq 0$ of jump from any state $n\in\N$ to the
absorption point $\d$. Such models have been studied in~\cite{vanDoorn2012,ChampagnatVillemonais2014} with an assumption of uniformly
bounded catastrophe rate, which we relax here.

In this example, we restrict for simplicity to the logistic cases, where there exist constants $\bar{b}>0$, $\underline{c}>0 $ and $\delta\in(0,1)$ such that, for $n$ large enough,
\begin{align}
\label{eq:cata}
b_n\leq n\bar{b},\quad d_n\geq n^2\underline{c}\quad\text{and}\quad a_n \leq \delta \underline{c}n.
\end{align}
This simple situation allows a similar computation as in Section~\ref{sec:one} (with $\gamma=1$ and $\beta_1=\beta_2=0$). Of course, the arguments can be easily adapted to any other cases where explicit
Lyapunov functions are known.

\begin{prop}
\label{prop:ex1}
Assume that~\eqref{eq:cata} holds for $n$ sufficiently large, then there exist constants $C,\lambda>0$ such that~\eqref{eq:conv} holds true.
\end{prop}

\begin{proof}
We use the same Lyapunov function $V_\varepsilon$ as in Section~\ref{sec:one} and we obtain
\begin{align*}
\mu_t(LV_\varepsilon)&=\sum_{n\in\N} \mu_t(n)\left(\frac{b_n}{(n+1)^{1+\varepsilon}}-\frac{c_n}{n^{1+\varepsilon}}-a_n V_\varepsilon(n)\right)\\
&\leq \sum_{n\in\N} \mu_t(n)\left(\bar{b}n^{-\varepsilon}-\underline{c}n^{1-\varepsilon}-a_n V_\varepsilon(n)\right).
\end{align*}
In addition
\begin{align*}
-\mu_t(V_\varepsilon)\mu_t(L\mathbbm{1}_\N)\leq \|V_\varepsilon\|_\infty \sum_{n\in\N} \mu_t(n)\left(\mathbbm{1}_{n=1} c_1+a_n\right).
\end{align*}
Hence there exists a constant $C>0$ such that
\begin{align*}
\mu_t(LV_\varepsilon)-\mu_t(V_\varepsilon)\mu_t(L\mathbbm{1}_\N)&\leq C-\underline{c}\sum_{n\in\N} \mu_t(n) n^{1-\varepsilon}+\sum_{n\in\N} \mu_t(n) a_n \left(\|V_\varepsilon\|_\infty-V_\varepsilon(n)\right)\\
&\leq C-\underline{c}\sum_{n\in\N} \mu_t(n) n^{1-\varepsilon}\left(1-\frac{\delta}{\varepsilon}\right).
\end{align*}
Taking $\varepsilon\in(\delta,1)$, we can conclude as in Section~\ref{sec:one}.
\end{proof}

\bigskip\noindent\textbf{Multi-dimensional birth and death processes with catastrophes.}
We now study the multi-dimensional case with catastrophes, for which, as far as we know, no result on quasi-stationary distributions is known. We consider the same model as in Section~\ref{sec:one} with an additional jump rate $a(n)\geq 0$ from any state $n\in\N^r$ to $\d$. The next result can be proved by an easy combination of the arguments of the proofs of Theorem~\ref{thm:one} and Proposition~\ref{prop:ex1}.
\begin{prop}
Under Assumption~\ref{hyp1} and the assumption
\begin{align*}
a(n)\ll c_{ii}(n) |n|^\gamma,
\end{align*}
there exist constants $C,\lambda>0$ such that~\eqref{eq:conv} holds true.
\end{prop}

\bigskip\noindent\textbf{Multi-dimensional birth and death processes with multiple births.}
We consider the same model as in Section~\ref{sec:one} with an additional feature: we allow multiple progeny at each birth time. To do so we consider, for all $n\in\N^r$, a probability measure
\begin{align*}
p_n:=\sum_{k\in\Z_+^r} p_{n,k}\delta_k.
\end{align*}
Then, when a birth occurs (at rate $b(n)$) in a population $n$, the new state of the population is $n+k$ with probability $p_{n,k}$.
A one-dimensional case has already been studied in~\cite{champagnat-claisse-15}.

\begin{prop}
Under Assumption~\ref{hyp1} and the assumption
\begin{align*}
M:=\sup_{n\in\N^r} \sum_{k\in\Z_+^r} |k|\,p_{n,k}\ <\ \infty,
\end{align*}
there exist constants $C,\lambda>0$ such that~\eqref{eq:conv} holds true.
\end{prop}

\begin{proof}
  The only term which differs from the proof of Theorem~\ref{thm:one} is the birth term in $\mu_t(LV_\varepsilon)$, given by
  \begin{multline*}
    \sum_{n\in\NN^r}\mu_t(n)\sum_{i=1}^rn_ib_i(n)\sum_{k\in\ZZ_+^r}p_{n,k}\sum_{j=|n|+1}^{|n|+|k|}\frac{1}{j^{1+\varepsilon}} \\
    \begin{aligned}
      & \leq\sum_{n\in\NN^r}\mu_t(n)\sum_{i=1}^rn_ib_i(n)\sum_{k\in\ZZ_+^r}\frac{|k|}{|n|^{1+\varepsilon}}p_{n,k} \\
      & \leq M\sum_{n\in\NN^r}\mu_t(n)\bar b |n|^{\beta_1-\varepsilon}.
    \end{aligned}
  \end{multline*}
  Thus, we obtain a similar bound as in Section~\ref{sec:one} for this term, and the proof can be completed as there.
\end{proof}

\bibliographystyle{abbrv}
\bibliography{biblio-bio,biblio-denis,biblio-math,biblio-math-nicolas}

\end{document}